\documentclass[aop,preprint]{imsart}

\usepackage{a4}
\usepackage{wrapfig}
\usepackage{boxedminipage}
\usepackage{pifont}
\usepackage{euscript}
\usepackage{boxedminipage}
\usepackage{dcolumn}
\usepackage{float}
\usepackage{fancybox}
\usepackage{rotating}
\usepackage{multirow}
\usepackage{latexsym}
\usepackage{amssymb}
\usepackage{amsmath}
\usepackage{epsfig}
\usepackage{color}
\usepackage{verbatim}
\usepackage{subfigure}
\usepackage{mathrsfs}
\usepackage{url}
\usepackage{lscape}
\usepackage{psfrag}

\usepackage{graphicx}

\usepackage{mparhack} 

\newtheorem{alg}{Algorithm}
\newtheorem{example}{Example}

\newtheorem{remark}{Remark}

\newtheorem{proposition}{Proposition}

\newcommand{\ealg}{\end{alg}}
\newcommand{\balg}{\begin{alg}}

\newcommand{\bG}{\mathbf{G}}

\newcommand{\T}{{\top}}

\newcommand{\bu}{\mathbf{u}}
\newcommand{\bU}{\mathbf{U}}

\newcommand{\ben}{\begin{enumerate}}
\newcommand{\een}{\end{enumerate}}
\newcommand{\beq}{\begin{equation}}
\newcommand{\eeq}{\end{equation}}
\newcommand{\ei}{\end{itemize}}
\newcommand{\bi}{\begin{itemize}}
\newcommand{\bex}{\begin{example}}
\newcommand{\eex}{\end{example}}
\newcommand{\berem}{\begin{remark}}
\newcommand{\erem}{\end{remark}}
\newcommand{\beprop}{\begin{proposition}}
\newcommand{\eprop}{\end{proposition}}

\newcommand{\Var}{\mathrm{Var}}
\newcommand{\var}{\Var}

\newcommand{\Cov}{\mathrm{Cov}}
\newcommand{\cov}{\Cov}

\renewcommand{\epsilon}{\varepsilon}
\renewcommand{\rho}{\varrho}
\renewcommand{\log}{\ln}
\renewcommand{\hat}{\widehat}
\renewcommand{\leq}{\leqslant}
\renewcommand{\geq}{\geqslant}

\newcommand{\Em}{\mathbb E}
\newcommand{\Pm}{\mathbb P}

\newcommand{\Z}{\mathbb Z}

\newcommand{\gvn}{\,|\,}

\newcommand{\ds}{\displaystyle}

\newcommand{\vect}[1]{\boldsymbol #1}

\newcommand{\di}{\mathrm{d}}

\newcommand{\ra}{\rightarrow}

\def\acro#1#2{\vskip4pt\hbox to\textwidth{\normalsize
\hbox to5pc{#1\hfill}\vtop{\advance\hsize by
-5pc\raggedright\noindent#2}}}

\def\symbol#1#2{\vskip4pt\hbox to\textwidth{\normalsize
\hbox to5pc{#1\hfill}\vtop{\advance\hsize by
-5pc\raggedright\noindent#2}}}

\def\distro#1#2{\vskip4pt\hbox to\textwidth{\normalsize
\hbox to5pc{#1\hfill}\vtop{\advance\hsize by
-5pc\raggedright\noindent#2}}}

\newcommand{\diag}{\mathrm{diag}}

\newcommand{\chk}[1]{}

\RequirePackage[OT1]{fontenc}
\RequirePackage{amsthm,amsmath}
\RequirePackage[numbers]{natbib}
\RequirePackage[colorlinks,citecolor=blue,urlcolor=blue]{hyperref}

\startlocaldefs
\numberwithin{equation}{section}
\theoremstyle{plain}
\newtheorem{thm}{Theorem}[section]
\newtheorem{lemma}{Lemma}[section]
\endlocaldefs

\begin{document}

\begin{frontmatter}
\title{Computing the  Drift of Random Walks in \\ Dependent Random Environments}
\runtitle{Computing the  drift of RWREs}


\begin{aug}
\author{\fnms{Werner R.W.}
  \snm{Scheinhardt}\ead[label=e1]{w.r.w.scheinhardt@utwente.nl}}
\and
\author{\fnms{Dirk P.} \snm{Kroese}
\ead[label=e2]{kroese@maths.uq.edu.au}
\ead[label=u1,url]{http://wwwhome.math.utwente.nl/~scheinhardtwrw/}
\ead[label=u2,url]{http://www.maths.uq.edu.au/~kroese/}
}

\runauthor{W.R.W. Scheinhardt and D.P. Kroese}

\affiliation{University of Twente and The University of Queensland}

\address{Department of Applied Mathematics\\
University of Twente\\
P.O. Box 217 \\
7500 AE  Enschede\\
The Netherlands\\
\printead{e1}\\
\printead{u1}
}

\address{School of Mathematics and Physics\\
The University of Queensland\\
Brisbane 4072\\ Australia \\
\printead{e2}\\
\printead{u2}}
\end{aug}

\begin{abstract}

Although the theoretical 
behavior of one-dimensional random walks in random
environments is well understood,   
the numerical evaluation of various characteristics of such processes
has received relatively little attention. 
This paper develops new theory and methodology for the computation of
the drift of the random walk for various dependent random
environments, including $k$-dependent and moving average environments.

\end{abstract}



\end{frontmatter}

 \section{Introduction} \label{sec:intro}

Random walks in random environments (RWREs) are well-known mathematical
models for motion 
through disorganized (random) media. They generalize
 ordinary random walks, usually on the $d$-dimensional lattice $\Z^d$, 
 via a two-stage random procedure.  First, the environment 
is generated according to some probability
 distribution (e.g., on a set ${\cal U}^{\, \Z}$, where ${\cal U}$ is the set
 of all possible environment states at  any position). Second, the
 walker performs an ordinary  random walk $\{X_n, n = 0,1, \ldots\}$ in which the transition
 probabilities at any state are
 determined by the environment at that state. RWREs exhibit
 interesting and unusual behavior that is not seen in ordinary random
 walks. For example, the walk can tend to infinity almost surely,  
  while the  
 speed (also called drift) is 0; that is, $\Pm(\lim_{n \ra \infty} X_n
 = \infty) = 1$, while $\lim_{n \ra \infty} X_n/n = 0$. The reason for
 such surprising behavior  is that RWREs can spend a long time in
 (rare) regions from which it is difficult to escape --- in effect, the
 walker becomes ``trapped'' for a long time.

Since the late 1960s a vast body of knowledge has been built up on the
behavior of RWREs. 
Early applications can be found in 
Chernov \cite{chernov67} and Temkin \cite{temkin69}; see also
Kozlov \cite{kozlov85} and references therein.
Recent applications to charge transport in designed materials 
are given in 
Brereton et al.\ \cite{brereton12} and 
Stenzel et al.\ \cite{stenzel14}. 
The mathematical framework for RWREs was laid by 
Solomon \cite{solomon75}, who proved conditions for 
recurrence/transience for one-dimensional RWREs and also derived law
of large number properties for such processes.   
Kesten et al.\ \cite{kesten75}  were the first to establish 
central limit-type 
scaling laws for transient RWREs,  and   
Sinai \cite{sinai83} proved such results for the recurrent case,
showing remarkable ``sub-diffusive'' behavior. 
Large deviations for these processes were obtained in Greven and Den Hollander 
\cite{greven94}. The main focus in these papers was on one-dimensional
random walks in independent
environments. 
Markovian environments were investigated in Dolgopyat
\cite{dolgopyat2008} and 
Mayer-Wolf et al. \cite{wolf04}.
Alili \cite{alili99} showed that in the one-dimensional case much of the
theory for independent environments 
could be 
generalized to the case where the environment process is stationary and ergodic. 
Overviews of the current state of the art, with a  focus on
higher-dimensional RWREs, can be found, for example, in 
Hughes \cite{hughes2}, Sznitman \cite{sznitman04}, 
Zeitouni \cite{zeit04,zeit2012}, and R\'ev\'esz \cite{revesz}.  

Although the theoretical behavior of one-dimensional RWREs is nowadays
well understood (in terms of transience/recurrence, law of large numbers, central limits, and 
large deviations), it remains difficult to find easy to
compute expressions  for key measures such as the drift of the
process. To the best of our knowledge such expressions are only
available in simple one-dimensional cases with independent random
environments. The purpose of this paper is to develop theory and
methodology for the computation of the drift of the random walk for
various dependent environments, including one where the environment
is obtained as a moving average of independent environments.

The rest of the paper is organized as follows. In
Section~\ref{sec:model} we formulate the model for a one-dimensional
RWRE in a stationary and ergodic environment and review
some  of the key results from \cite{alili99}. We then formulate
special cases for the environment: the iid, the Markovian, the
$k$-dependent, and the moving average environment.
In Section~\ref{sec:eval} we derive explicit (computable) 
expressions for the drift for each of these models, using a novel
construction involving an auxiliary Markov chain. Conclusions and
directions for future research are given in Section~\ref{sec:concl}.


\section{Model and preliminaries}\label{sec:model}

In this section we review some key results on one-dimensional 
RWREs and introduce the class of ``swap-models''
that we will study in more detail. We mostly follow the notation of Alili \cite{alili99}.

\subsection{General theory} 

Consider a stochastic
process $\{X_n,n=0,1,2,\ldots\}$ with state space $\Z$, and a stochastic 
``Underlying'' environment $\bU$ taking values in some set ${\cal
  U}^{\,\Z}$, where ${\cal U}$ is the set of possible environment states
for each site in $\Z$. We assume that $\bU$ is stationary (under $\Pm$)
as well as  ergodic (under the natural shift operator on $\Z$). 
The evolution of $\{X_n\}$ depends on the realization of $\bU$, which is random
but fixed over time. For any realization $\bu$ of $\bU$ the process $\{X_n\}$ behaves as a simple random walk with transition probabilities 
\begin{equation} 
\begin{split} \label{transitions}
\Pm(X_{n+1} &= i+1 \gvn X_n = i, \bU = \bu) = \alpha_i(\bu)\\
\Pm(X_{n+1} &= i-1 \gvn X_n = i, \bU = \bu) =
\beta_i(\bu)=1-\alpha_i(\bu).
\end{split}
\end{equation}

The general behavior of $\{X_n\}$ is well
understood. Theorems~\ref{thm:alili1} and \ref{thm:alili2} below 
completely describe the transience/recurrence behavior and the Law of
Large Numbers behavior of $\{X_n\}$. The key quantities in these
theorems are given first.
Define  
\begin{equation} \label{rhodef}
\sigma_i = \sigma_i(\bu)=\frac{\beta_i(\bu)}{\alpha_i(\bu)}, \quad i \in \Z\;,
\end{equation} 
and let
\begin{equation} \label{Sdef}
S = 1 + \sigma_1 + \sigma_1 \,\sigma_2 + \sigma_1 \,\sigma_2 \,\sigma_3 + \cdots
\end{equation}
and 
\begin{equation}\label{Fdef}
F = 1 + \frac{1}{\sigma_{-1}} + \frac{1}{\sigma_{-1}\,\sigma_{-2}} +
\frac{1}{\sigma_{-1}\,\sigma_{-2}\, \sigma_{-3}} + \cdots 
\end{equation}

\begin{thm}~(Theorem 2.1 in  \cite{alili99}) \label{thm:alili1}
\begin{enumerate}
\item If $\Em [\log \sigma_0]< 0$, then almost surely $\ds \lim_{n\ra \infty} {X_n} = \infty\;.$
\item If $\Em [\log \sigma_0] > 0$, then almost surely $\ds \lim_{n\ra \infty} {X_n} =-\infty\;.$
\item If $\Em [\log \sigma_0] = 0$, then almost surely $\ds \liminf_{n\ra \infty} {X_n} =-\infty$ and $\ds \limsup_{n\ra \infty} {X_n} =\infty\;.$
\end{enumerate}
\end{thm}

\begin{thm}~(Theorem 4.1 in  \cite{alili99})  \label{thm:alili2}
\begin{enumerate}
\item If $\Em[ S] < \infty$, then almost surely $\ds \lim_{n\ra \infty} \frac{X_n}{n} =
  \frac{1}{\Em[(1+\sigma_0)S]} = \frac{1}{2\Em[S]-1}\;.$
\item If $\Em[ F] < \infty$, then almost surely $\ds \lim_{n\ra \infty} \frac{X_n}{n} =
   \frac{-1}{\Em[(1+\sigma_0^{-1})F]} = \frac{-1}{2\Em[F]-1}\;.$
\item If $\Em[ S] = \infty$ and $\Em[ F] = \infty$, then almost surely 
 $\ds \lim_{n\ra \infty} \frac{X_n}{n} = 0$.
\end{enumerate}
\end{thm}

Note that we have added the second equalities in statements 1.\ and
2.\ of Theorem~\ref{thm:alili2}. These follow directly from the stationarity of
$\bU$. In particular, if $\theta$ denotes the shift operator on $\Z$,
then
\begin{align*}
\Em[\sigma_0\sigma_1 \cdots \sigma_{n-1}] 
&= \Em\left[\frac{\beta_0(\bU)\beta_1(\bU)\cdots\beta_{n-1}(\bU)}{\alpha_0(\bU)\alpha_1(\bU)\cdots\alpha_{n-1}(\bU)} \right]\\
&= \Em\left[\frac{\beta_1(\theta\bU)\beta_2(\theta\bU)\cdots\beta_{n}(\theta\bU)}{\alpha_1(\theta\bU)\alpha_2(\theta\bU)\cdots\alpha_{n}(\theta\bU)} \right]\\
(\mbox{apply } \theta\bU \stackrel{d}{=}\bU)\ &= \Em\left[\frac{\beta_1(\bU)\beta_2(\bU)\cdots\beta_{n}(\bU)}{\alpha_1(\bU)\alpha_2(\bU)\cdots\alpha_{n}(\bU)} \right]\\
&=\Em[\sigma_1\sigma_2 \cdots \sigma_{n}],
\end{align*}
from which it follows  that $\Em[(1+\sigma_0)S] = 2\Em[S]-1$. 

We will call $\lim_{n\ra \infty} X_n/n$ the {\em drift} of the
process $\{X_n\}$, and denote it by $V$.
Note that, as mentioned in the  introduction,  it is possible for the chain to be transient with drift 0 (namely when  $\Em [\log \sigma_0] \neq 0$,  $\Em[ S] = \infty$ and $\Em[ F] = \infty$).

\subsection{Swap model}\label{ssec:swapmodel}
We next focus on what we will call {\em swap} models. 
Here,  ${\cal U}=\{-1,1\}$; that is, we
assume that all elements $U_i$ of the process $\bU$ take value either
$-1$ or $+1$. We assume that the transition probabilities in
state $i$ only depends on $U_i$, and not on other elements of $\bU$, as
follows. When $U_i=-1$, the transition probabilities of $\{X_n\}$ from
state $i$ to states $i+1 $ and $i-1$ are swapped with respect to the
values they have when $U_i=+1$. Thus, for some fixed value $p$ in
$(0,1)$ we let $\alpha_i(\bu)=p$ (and $\beta_i(\bu)=1-p$) if $u_i =
1$, and $\alpha_i(\bu)=1-p$ (and $\beta_i(\bu)=p$) if $u_i =
-1$. Thus, (\ref{transitions}) becomes
\[
\Pm(X_{n+1} = i+1 \gvn X_n = i, \bU = \bu) = \begin{cases}
p & \text{ if } u_i = 1 \\
1-p & \text{ if } u_i = -1
\end{cases}
\]
and 
\[
\Pm(X_{n+1} = i-1 \gvn X_n = i, \bU = \bu) = \begin{cases}
1-p & \text{ if } u_i = 1 \\
p & \text{ if } u_i = -1\;.
\end{cases}
\]
Notice that due to our convenient choice of notation for the states in
${\cal U}=\{-1,1\}$ we have 
\[
\sigma_i = \frac{p}{1-p} \mathbb{I}(U_i=-1)+\frac{1-p}p \mathbb{I}(U_i=1) =  \sigma^{U_i},
\]
where $\sigma =(1-p)/p$.  Also, for the quantities in Theorems~\ref{thm:alili1} and~\ref{thm:alili2} we find the following.
\begin{equation} \label{logrho-swap}
\Em[\log \sigma_0]=\Em[U_0 \log\sigma]=\log\sigma\ \Em[U_0],
\end{equation}
the sign of which (and hence the a.s.\ limit of $X_n$) only depends on
whether $p$ is less than or greater than 1/2, and on whether $\Em[U_0]$ is positive or negative, regardless of the dependence structure between the $\{U_i\}$. Furthermore,
\begin{equation} \label{ES-swap}
\Em[S]= 
\sum_{n=0}^\infty \Em\left[\sigma^{\sum_{i=1}^n U_{i}}\right]\quad \mbox{ and } \quad 
\Em[F]= \sum_{n=0}^\infty \Em\left[\sigma^{-\sum_{i=1}^n U_{-i}}\right]\;.
\end{equation}
In what follows we will focus on $\Em[S]$, since analogous results for
$\Em[F]$ follow by replacing $\sigma$ with $\sigma^{-1}$ and $p$ with
$1-p$. This follows from the stationarity of $\bU$, which implies that for any $n$ the product $\sigma_{-1}\sigma_{-2}\cdots \sigma_{-n}$ has the same distribution as $\sigma_1\sigma_2\cdots \sigma_n$ (apply a shift over $n+1$ positions).

Next,  we need to choose a dependence structure for $\bU$. 
The standard case, first studied  by 
Sinai \cite{sinai83}, simply assumes that the $\{U_i\}$ are iid
  (independent and identically distributed):

\medskip
\noindent{\bf Iid environment.}
Let the $\{U_i\}$ be iid with 
\[
\Pm(U_{i}=1)=\alpha, \qquad \Pm(U_{i}=-1)=1-\alpha
\]
for some $0 < \alpha < 1$. In this case the model has two parameters: $\alpha$ and $p$.

\medskip
\noindent
We  extend this to a more general model where $\bU$ is
generated by a stationary and ergodic Markov chain $\{Y_i, i \in
\Z\}$ taking values in a finite set $\{1,\ldots,m\}$. In particular, we let
$U_i = g(Y_i)$, where $g:\{0,\ldots,m\} \ra \{-1,1\}$ is a given
  function. Despite its simplicity, this formalism covers a number of
  interesting dependence structures on $\bU$, discussed next.

\medskip
\noindent{\bf Markov environment.} Let $U_i = Y_i$, where $\{Y_i\}$ is
a stationary discrete-time Markov chain on $\{-1,1\}$, with one-step
transition matrix $P$ given by
\[
P=\left[
\begin{array}{cc}
1-a&a\\
b&1-b
\end{array}
\right],
\]
for some $a,b \in (0,1)$. The $\{U_i\}$ form a dependent Markovian
environment depending on $a$ and $b$.

\medskip

\noindent{\bf $k$-dependent environment.} \label{sec:nstep}
Let $k \geq 1$ be a fixed integer. Our goal is to obtain a generalization of the Markovian environment in which 
the conditional distribution of $U_i$ given all other variables
is the same as the conditional distribution of $U_i$ given only 
$U_{i-k}, \dots, U_{i-1}$ (or, equivalently, given 
$U_{i+1},\ldots,U_{i+k}$).
To this end we define a $k$-dimensional Markov chain $\{Y_i,i\in \Z\}$ on $\{-1,1\}^k$ as follows.
 From any state $(u_{i-k},  \ldots, u_{i-1})$ in $\{-1,1\}^k$,  $\{Y_i\}$ has two possible one-step transitions,  given by
\[
(u_{i-k},  \ldots, u_{i-1})
\rightarrow
({u}_{i-k+1},\ldots, u_{i-1}, u_i), \quad u_j \in \{-1,1\}, 
\]
with corresponding probabilities $1-a_{({u}_{i-k},\ldots,{u}_{i-2})}$, $a_{({u}_{i-k},\ldots,{u}_{i-2})}$, $b_{({u}_{i-k},\ldots,{u}_{i-2})}$, and $1-b_{({u}_{i-k},\ldots,{u}_{i-2})}$, 
for $(u_{i-1},u_i)$ equal to $(-1,-1),$  $(-1,1),$ $(1,-1)$, and $(1,1)$,
respectively.  Now let  $U_i$ denote the last component of $Y_i$. 
Then $\{U_i, i \in \Z\}$ is a $k$-dependent environment, and  $Y_i=(U_{i-k+1},  \ldots, U_{i})$.

Note the correspondence in notation with the (1-dependent) Markov environment: $a$ indicates transition probabilities from $U_{i-1}=-1$ to $U_{i}=+1$, and $b$ from $U_{i-1}=+1$ to $U_{i}=-1$, where in both cases the subindex denotes the dependence on $U_{i-k}, \ldots, U_{i-2}$.
%

\medskip
\noindent {\bf Moving average environment.}  \label{sec:MAS}
Consider a ``moving average'' environment, which is built up in two phases as follows. First, start with an
iid environment $\{\hat U_i\}$ as in the iid case, with $\Pm(\hat U_{i}=1)=\alpha$. Let $Y_i =
(\hat U_i,\hat U_{i+1},\hat U_{i+2})$. Hence, $\{Y_i\}$ is a Markov process 
with states $1= (-1,-1,-1), 2 = (-1,-1,1), \ldots, 8 = (1,1,1)$ 
(lexicographical order). The corresponding transition matrix clearly is given by
\begin{equation}   \label{P_movav}
P  = \begin{bmatrix}
1-\alpha & \alpha & 0 & 0 & 0 & 0 & 0 & 0\\
0 & 0 & 1-\alpha & \alpha & 0 & 0 & 0 & 0 \\
0 & 0 & 0 & 0 & 1-\alpha & \alpha & 0 & 0 \\
0 & 0 & 0 & 0 & 0& 0& 1-\alpha & \alpha \\
1-\alpha & \alpha & 0 & 0 & 0 & 0 & 0 & 0\\
0 & 0 & 1-\alpha & \alpha & 0 & 0 & 0 & 0 \\
0 & 0 & 0 & 0 & 1-\alpha & \alpha & 0 & 0 \\
0 & 0 & 0 & 0 & 0& 0& 1-\alpha & \alpha \\
\end{bmatrix}.
\end{equation}
Now define $U_i = g(Y_i)$, where $g(Y_i)=1$ if at least two of the three random variables $\hat U_i,\hat U_{i+1}$ and 
$\hat U_{i+2}$ are 1, and $g(Y_i)=-1$ otherwise. Thus,
\begin{equation}   \label{g_movav}
(g(1), \ldots, g(8)) = (-1,-1,-1,1,-1,1,1,1)\;,
\end{equation}
and we see that each $U_i$ is obtained by taking the moving average of 
$\hat{U}_i, \hat{U}_{i+1}$ and $\hat{U}_{i+2}$, as illustrated in 
Figure~\ref{fig:movav}.

\begin{figure}[H]
\centering
 \includegraphics[width=0.7\linewidth]{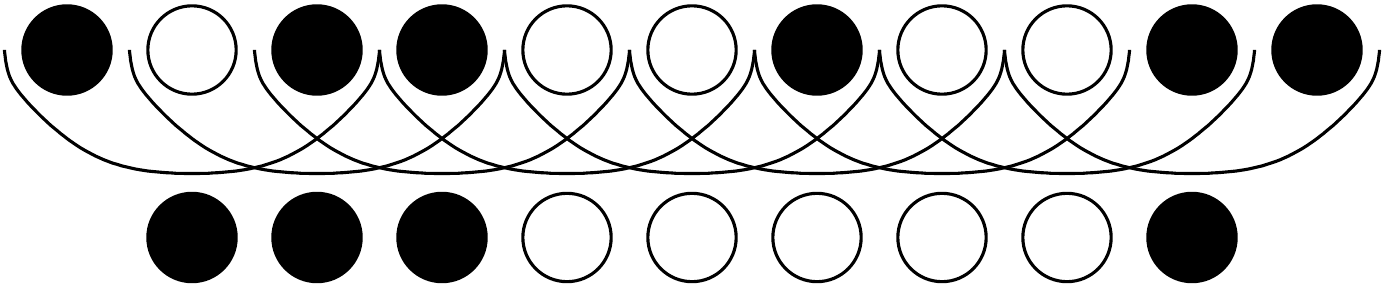}
\label{fig:movav}
\caption{Moving average environment.}
\end{figure}

\section{Evaluating the drift}\label{sec:eval}

As a starting point for the analysis, we begin in Section~\ref{ssec:iid} with the solution for the iid
environment, based on first principles. As mentioned earlier, 
this case was first studied  by 
Sinai \cite{sinai83}.
Then, in 
Section~\ref{sec:generalsolution} we give the general solution approach for the Markov-based swap model, followed by sections with results on the transience/recurrence and on the drift for the random environments mentioned in Section~\ref{ssec:swapmodel}: the Markov environment, the 2-dependent environment, and the moving average environment (all based on Section~\ref{sec:generalsolution}).

\subsection{Iid environment}\label{ssec:iid}
As a warm-up we consider the iid case first, with $\Pm(U_{i}=1)=\alpha=1-\Pm(U_{i}=-1)$.
Here,  
\[
\Em [\log \sigma_0]=\Em [U_0] \log \sigma = (1-2\alpha) \log\frac{1-p}{p}.
\]
Hence, by Theorem~\ref{thm:alili1}, we have the following findings, consistent with intuition.
$X_n\rightarrow +\infty$ a.s.\ if and only if either $\alpha>1/2$ and $p > 1/2$, or $\alpha<1/2$ and $p<1/2$; 
$X_n\rightarrow -\infty$ a.s.\  if and only if either $\alpha>1/2$ and $p < 1/2$, or $\alpha<1/2$ and $p>1/2$; and 
$\{X_n\}$ is recurrent a.s.\ if and only if either $\alpha=1/2$, or $p = 1/2$, or both. 

Moving on to Theorem~\ref{thm:alili2}, we have
\begin{align}
\Em [S]&=\sum_{n=0}^\infty \Em\left[\sigma^{\sum_{i=1}^n U_{i}}\right] =
\sum_{n=0}^\infty \left(\Em[\sigma^{U_1}]\right)^n = \sum_{n=0}^\infty
(\sigma^{-1}(1-\alpha) + \sigma \alpha)^n, \label{iidS}
\end{align}
which is finite if and only if $\sigma^{-1}(1-\alpha) + \sigma \alpha<1$; that is,
$\Em [S]<\infty$ if and only if either $\alpha>1/2$ and $p \in (1/2, \alpha)$, or $\alpha<1/2$ and $p \in (\alpha, 1/2)$. 
Similarly (replace $\sigma$ by $\sigma^{-1}$ and $p$ by $1-p$), $\Em [F] = \sum_{n=0}^\infty
(\sigma(1-\alpha) + \sigma^{-1} \alpha)^n <\infty$ if and only if either $\alpha>1/2$ and $p \in (1-\alpha,1/2)$, or $\alpha<1/2$ and  $p \in (1/2, 1-\alpha)$.

Clearly the cases with respect to  $\Em[S]$  and $\Em[F]$ do not entirely cover the cases we concluded to be transient above. E.g., when  $\alpha > 1/2$ and $p \in [\alpha,1]$, the process tends to $+\infty$, but the drift is zero. We summarize our findings in the following theorem.

\begin{thm} \label{thm:iid}~
We distinguish between transient cases with and without drift, and the recurrent case as follows.
\begin{enumerate}
\item[1a.] If either $\alpha > 1/2$ and $p\in (1/2, \alpha)$ or $\alpha < 1/2$ and $p
  \in (\alpha, 1/2)$, then almost surely $\ds \lim_{n\ra \infty} {X_n} = \infty\;$ and
\begin{equation} \label{iid}
V = (2p-1)\frac{\alpha-p}{\alpha(1-p)+(1-\alpha)p} > 0\;.
\end{equation}
\item[1b.] If either 
 $\alpha > 1/2$ and $p \in (1-\alpha,1/2)$ or 
  $\alpha < 1/2$ and $p \in (1/2, 1-\alpha)$, then  almost surely $\ds \lim_{n\ra \infty} {X_n} = -\infty\;$ and 
 \begin{equation} \label{iid2}
V = - (1-2p)\frac{\alpha-(1-p)}{\alpha p+(1-\alpha)(1-p)} < 0\;.
\end{equation}
\item[2a.] If either $\alpha > 1/2$ and $p\in [\alpha,1]$ or $\alpha < 1/2$ and $p
  \in [0, \alpha]$, then almost surely $\ds \lim_{n\ra \infty} {X_n} = \infty\;,$ but $V=0$.
\item[2b.] If either $\alpha > 1/2$ and $p\in [0, 1-\alpha]$ or $\alpha < 1/2$ and $p
  \in [1-\alpha, 1]$, then almost surely $\ds \lim_{n\ra \infty} {X_n} = -\infty\;,$ but $V=0$.
\item[3.] Otherwise (when $\alpha = 1/2$ or $p=1/2$ or both),  $\{X_n\}$ is recurrent and $V = 0$. 
\end{enumerate}
\end{thm}
\begin{proof}
Immediate from the above; (\ref{iid}) follows from (\ref{iidS})
 by using $\sigma=(1-p)/p$;  and similar for \eqref{iid2}. 
\end{proof}
\medskip

We illustrate the drift as a function of $\alpha$ and $p$ in Figure~\ref{fig:iidcase}.

\begin{figure}[H]
\begin{center}
\includegraphics[width=0.7\linewidth,clip=]{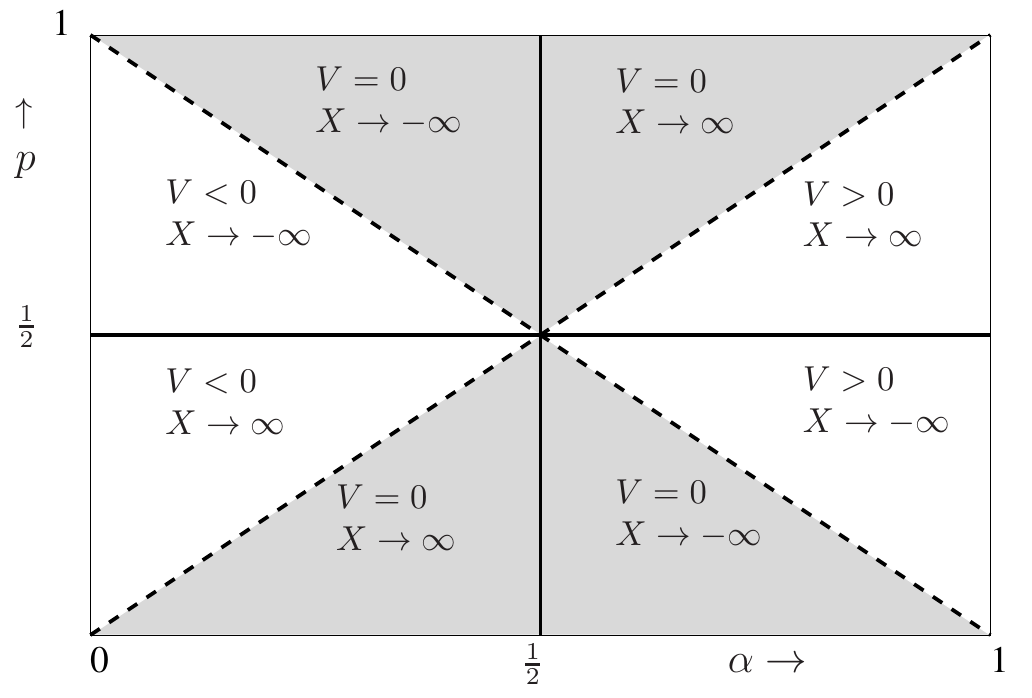}
\caption{Graphical representation of Theorem~\ref{thm:iid}. Solid
  lines, where the process is recurrent, divide the remaining
  parameter space in four quadrants. In quadrants I and III, $\{X_n\}$
  moves to the right; in quadrants II and IV, $\{X_n\}$ moves to the left. In gray areas (including dashed boundaries and boundaries at $p=0, 1$), the drift is zero. In white areas (including boundaries at $\alpha=0, 1$), the drift is nonzero.} \label{fig:iidcase}
\end{center}
\end{figure}

\subsection{General solution for swap models}  \label{sec:generalsolution}

Consider a RWRE swap model with a random environment generated by a Markov
chain $\{Y_i, i \in \Z\}$, as specified in
Section~\ref{sec:model}. We already saw that the a.s.\ limit of $X_n$ only depends on whether $p$ is less than or larger than 1/2, and on whether $\Em[U_0]$ is positive or negative, regardless of the dependence structure between the $\{U_i\}$, see (\ref{logrho-swap}). The other key quantity to evaluate is (see
\eqref{ES-swap}): 
\[
\Em[S] = 
\sum_{n=0}^\infty \Em\left[\sigma^{\sum_{i=1}^n U_{i}}\right] = 
\sum_{n=0}^\infty \Em\left[\sigma^{\sum_{i=1}^n g(Y_{i})}\right]\;.
\] 
 Let 
\[
G^{(n)}_{y}(\sigma) = \Em\left[ \sigma^{\sum_{i=1}^n g(Y_i)} \gvn Y_0 = y\right],
\quad y = 1,\ldots,m\;.
\]
Let $P = (P_{y, y'})$ be the one-step transition matrix of $\{Y_i\}$.
Then, by conditioning on $Y_1$,  
\[
\begin{split}
G^{(n+1)}_y(\sigma) & =  \Em\left[ \sigma^{\sum_{i=1}^{n+1} g(Y_i)} \gvn Y_0 =
  y\right] =  \Em\left[ \sigma^{\sum_{i=2}^{n+1} g(Y_i)}
  \sigma^{g(Y_1)} \gvn Y_0   = 
  y\right] \\
& = \sum_{y'=1}^m P_{y,y'} \sigma^{g(y')} G_{y'}^{(n)}(\sigma)\;.
\end{split}
\]
In matrix notation, with $\bG^{(n)}(\sigma) =(G_{1}^{(n)}(\sigma), \ldots,  G_{m}^{(n)}(\sigma))^\T$,
we can write this as 
\[
\bG^{(n+1)}(\sigma) = P D \bG^{(n)}(\sigma),
\]
where 
\[
D = \diag(\sigma^{g(1)}, \ldots,\sigma^{g(m)})\;.
\]
It follows, also using $G^{(0)}_y(\sigma)=1$, that
\[
\bG^{(n)}(\sigma)=(PD)^n \bG^{(0)}(\sigma)=(PD)^n\vect{1},
\]
where $\vect{1} = (1,1)^\T$, 
and hence
\begin{equation*}
\Em[S] =\sum_{n=0}^\infty \vect{\pi} \bG^{(n)}(\sigma)\ 
              =\ \vect{\pi}\sum_{n=0}^\infty (PD)^n \vect{1}, 
\end{equation*}
where $\vect{\pi}$ denotes the stationary distribution vector for $\{Y_i\}$. The matrix series $\sum_{n=0}^\infty (PD)^n$ converges if and only if $\mbox{Sp}(PD)<1$, where $\mbox{Sp}(\cdot)$ denotes the spectral radius, and in that  case the limit is $(I-PD)^{-1}$. Thus, we end up with 
\begin{equation}  \label{series}
\Em[S]=
 \begin{cases}
 \vect{\pi} (I-PD)^{-1}  \vect{1} & \text{ if  Sp}(PD)<1\\
\infty & \text{ else. }
\end{cases}
\end{equation}

Based on the above, the following subsections will give results on the
transience/recurrence and on the drift for the random environments
mentioned in Section~\ref{ssec:swapmodel}.

\subsection{Markov environment}     \label{sec:markovdependence}
The quantity $\Em[\log \sigma_0]$ in Theorem~\ref{thm:alili1}, which determines whether $X_n$ will diverge to $+\infty$ or $-\infty$, or is recurrent, is given by 
\[
\Em[\log \sigma_0]=\frac{b}{a+b}\log \sigma^{-1} + \frac{a}{a+b}\log\sigma =\frac{a-b}{a+b}\log\frac{1-p}{p}.
\]
Hence,
$X_n\rightarrow +\infty$ a.s.\ if and only if either $a > b$ and $p > 1/2$, or $a < b$ and $p<1/2$; 
$X_n\rightarrow -\infty$ a.s.\  if and only if either $a > b$ and $p < 1/2$, or $a < b$ and $p>1/2$; and 
$\{X_n\}$ is recurrent a.s.\ if and only if either $a = b$, or $p = 1/2$, or both. 

Next we study $\Em[S]$ to find the drift. In the context of Section~\ref{sec:generalsolution} the processes $\{U_i\}$ and $\{Y_i\}$ are identical and the function $g$ is the identity on the state space ${\cal U}=\{-1,1\}$. 
%
 %
Thus, the matrix  $D$ is given by $D=\diag(\sigma^{-1}, \sigma)$, and since $P$ is as in Section~\ref{ssec:swapmodel}, the matrix $PD$ is given by 
\[
PD=\left[
\begin{array}{cc}
(1-a)\sigma^{-1}&{a}\sigma\\
b\sigma^{-1}&(1-b)\sigma
\end{array}
\right],
\]
for which we have the following.

%

\begin{lemma}  \label{lem}
The matrix series $\sum_{n=0}^\infty (PD)^n$ converges to
\begin{equation} \label{seriesanswer}
(I-PD)^{-1}=\frac1{\det(I-PD)}\left[
\begin{array}{cc}
1-(1-b)\sigma&{a}\sigma\\
b\sigma^{-1}&1-{(1-a)}\sigma^{-1}
\end{array}
\right],
\end{equation}
with $\det(I -PD)=2-a-b-\left(\frac{1-a}\sigma+(1-b)\sigma\right)$,
iff $\sigma$ lies between 1 and~$\frac{1-a}{1-b}$.
\end{lemma}
Note that the condition that $\sigma$ lies between 1 and~$\frac{1-a}{1-b}$ can either mean 
$1<\sigma<\frac{1-a}{1-b}$ (when $a<b$), or $\frac{1-a}{1-b}<\sigma<1$ (when $a>b$).

\medskip
\begin{proof}
The series $\sum_{n=0}^\infty (PD)^n$ converges if and only if $\mbox{Sp}(PD)<1$, where $\mbox{Sp}(\cdot)$ denotes the spectral radius $\max_i  \gvn \lambda_i \gvn $. The eigenvalues $\lambda_1, \lambda_2$ follow from 
\[
 \gvn \lambda I -PD \gvn =\lambda^2-A\lambda+(1-a-b)=0, \qquad \mbox{where} \qquad A=(1-a)\sigma^{-1}+(1-b)\sigma.
\]
The discriminant of this quadratic equation is
\[
A^2-4(1-a)(1-b)+4ab=\left(\frac{1-a}\sigma-(1-b)\sigma\right)^2+4ab>0,
\]
so the spectral radius is given by the largest eigenvalue,
\[
\mbox{Sp}(PD)=\frac{A+\sqrt{A^2-4(1-a-b)}}{2}.
\]
Clearly $\mbox{Sp}(PD)<1$ if and only if $\sqrt{A^2-4(1-a-b)}<2-A$, or equivalently $A<2-a-b$.
Substituting the definition of $A$ and multiplying by $\sigma$ this leads to 
\[
(1-b)\sigma^2 -(2-a-b)\sigma +(1-a)<0,
\]
or equivalently,
\[
(\sigma-1)\big((1-b)\sigma-(1-a)\big)<0.
\]
Since the coefficient of $\sigma^2$ in the above is $1-b>0$, the statement of the lemma now follows immediately.
\end{proof}

This leads to the following theorem.

\begin{thm} \label{thm:general}~
We distinguish between transient cases with and without drift, and the recurrent case as follows.
\begin{enumerate}
\item[1a.] If either $a>b$ and $p\in \big(\frac12, \frac{1-b}{(1-a)+(1-b)}\big)$ or $a<b$ and $p\in  \big(\frac{1-b}{(1-a)+(1-b)},\frac12 \big)$, then almost surely $\ds \lim_{n\ra \infty} {X_n} = \infty\;$ and
\begin{equation} \label{Markovdrift}
V =(2p-1)
\frac{(1-b)(1-p)-(1-a)p}{\left(b+\frac{a-b}{a+b}\right)(1-p) + \left(a-\frac{a-b}{a+b}\right)p} > 0\;.
\end{equation}
\item[1b.] If either 
$a>b$ and $p\in  \big(\frac{1-a}{(1-a)+(1-b)},\frac12 \big)$ or $a<b$ and $p\in  \big(\frac12, \frac{1-a}{(1-a)+(1-b)} \big)$, then  almost surely $\ds \lim_{n\ra \infty} {X_n} = -\infty\;$ and 
 \begin{equation} \label{Markovdrift2}
V = -(1-2p)\frac{(1-b)p-(1-a)(1-p)}{\left(b+\frac{a-b}{a+b}\right)p + \left(a-\frac{a-b}{a+b}\right)(1-p)} < 0\;.
\end{equation}
\item[2a.] If either $a>b$ and $p\in  \big[\frac{1-b}{(1-a)+(1-b)},1 \big]$ or $a<b$ and 
$p\in \big[0, \frac{1-b}{(1-a)+(1-b)} \big]$, then almost surely $\ds \lim_{n\ra \infty} {X_n} = \infty\;,$ but $V=0$.
\item[2b.] If either $a>b$ and $p\in \big[0, \frac{1-a}{(1-a)+(1-b)} \big]$ or $a<b$ and 
$p \in  \big[\frac{1-a}{(1-a)+(1-b)},1 \big]$, then almost surely $\ds \lim_{n\ra \infty} {X_n} = -\infty\;,$ but $V=0$.
\item[3.] Otherwise (when $a=b$ or $p=1/2$ or both),  $\{X_n\}$ is recurrent and $V = 0$. 
\end{enumerate}
\end{thm}

\begin{proof}
Substitution of (\ref{seriesanswer}) and $\vect{\pi}=\frac1{a+b}(b,a)$ in (\ref{series}) leads to 
\begin{eqnarray*}
V^{-1}&=&2\Em[S]-1\\
&=& \frac2{(a+b)\det(I-PD)}\ (b,a) \left[
\begin{array}{cc}
1-(1-b)\sigma&a\sigma\\
b\sigma^{-1}&1-{(1-a)}\sigma
\end{array}
\right]
\left(\!\!\!
\begin{array}{c}
1\\1
\end{array}\!\!\!
\right)-1\\
&=& \frac2{\det(I-PD)} \frac{(a+b)-(1-a-b)(b\sigma+a\sigma^{-1})}{a+b}-1\\
&=& \frac{1+\sigma}{1-\sigma}
\frac{\left(b+\frac{a-b}{a+b}\right)\sigma +
  \left(a-\frac{a-b}{a+b}\right)}{(1-b)\sigma-(1-a)}\\
&=& \frac{1}{2p-1}
\frac{\left(b+\frac{a-b}{a+b}\right)(1-p) + \left(a-\frac{a-b}{a+b}\right)p}{(1-b)(1-p)-(1-a)p}.
\end{eqnarray*}
When $\sigma$ lies between $1$ and $\frac{1-a}{1-b}$, i.e. when $p=(1+\sigma)^{-1}$ lies between $1/2$ and $(1-b)/((1-a)+(1-b))$, it follows by Lemma~\ref{lem} that the process has positive drift, given by the reciprocal of the above. This proves (\ref{Markovdrift}). The proof of (\ref{Markovdrift2}) follows from replacing $\sigma$ by $\sigma^{-1}$ and $p$ by $1-p$, and adding a minus sign. The other statements follow immediately.
\end{proof}

When we take $a+b=1$ we obtain the iid case of the previous section, with $\alpha=a/(a+b)$. Indeed the theorem then becomes identical to Theorem~\ref{thm:iid}. In the following subsection we make a comparison between the Markov case and the iid case.



\subsubsection{Comparison with the iid environment}
To study the impact of the (Markovian) dependence, we reformulate the expression for the drift in Theorem~\ref{thm:general}. Note that the role of $\alpha$ in the iid case is played by $P(U_0=1)=a/(a+b)$ in the Markov case. Furthermore, we can show that the correlation coefficient between two consecutive $U_i$'s satisfies
\[
\rho\equiv \rho(U_0, U_1)=\frac{\cov(U_0, U_1)}{\var(U_0)}=\frac{\frac{a+b-4ab}{a+b}-\left(\frac{a-b}{a+b}\right)^2}{1-\left(\frac{a-b}{a+b}\right)^2}=1-a-b.
\]
So $\rho$ depends on $a$ and $b$ only through their sum $a+b$, with
extreme values 1 (for $a=b=0$; i.e., $U_i\equiv U_0$) and $-1$ (for
$a=b=1$; that is, $U_{2i}\equiv U_0$ and $U_{2i+1}\equiv -U_0$). The intermediate case $a+b=1$ leads to  $\rho=0$ and corresponds to the iid case, as we have seen before.
To express $V$ in terms of $\alpha$ and $\rho$ we solve the system of equations 
$\frac{a}{a+b}=\alpha$ and $1-a-b=\rho$, leading to the solution 
\begin{align*}
a&=(1-\rho)\alpha\\
b&=(1-\rho)(1-\alpha).
\end{align*}
Substitution in the expression for $V$ (here in case of positive drift only, see (\ref{Markovdrift})) and rewriting yields
\[
V =(2p-1)
\frac{\alpha-p + \rho(1-\alpha -p)}{\big(\alpha(1-p) + (1-\alpha)p\big) (1+\rho) - \rho}. 
\]
This enables us not only to immediately recognize the result
(\ref{iid}) for the iid case (take $\rho=0$), but also to study the
dependence of the drift $V$ on $\rho$. 
Note that due to the restriction that $a$ and $b$ are probabilities,
it must hold that $\rho > \max\{1 - 1/\alpha, 1 - 1/(1-\alpha)\}$.

Figures~\ref{fig:rwre3Vversusp} and~\ref{fig:rwre3Vversusrho}
illustrate 
various aspects of the difference between iid and Markov cases. 
Clearly, compared to the iid case (for the same value of $\alpha$),
the Markov case with positive correlation coefficient has lower drift,
but also a lower `cutoff value' of $p$ at which the drift becomes
zero. 
For negative correlation coefficients we see a higher cutoff value, but not all values of $\alpha$ are possible (since we should have $a<1$). 
Furthermore, for weak correlations the drift (if it exists) tends to be larger than for strong correlations (both positive and negative), depending on $p$ and $\alpha$. 
%
 Note that Figure~\ref{fig:rwre3Vversusrho} seems to suggest there are two cutoff values 
 in terms of the correlation coefficient. However, it should be realized that drift curves corresponding to some $\alpha$ are no longer drawn for negative correlations since the particular value of $\alpha$ cannot be attained. E.g., when $\rho$ is close to $-1$, then $a$ and $b$ are both close to 1, hence $\alpha$ can only be close to 1/2.

\begin{figure}[H]
\centering
\includegraphics[width=0.8\linewidth]{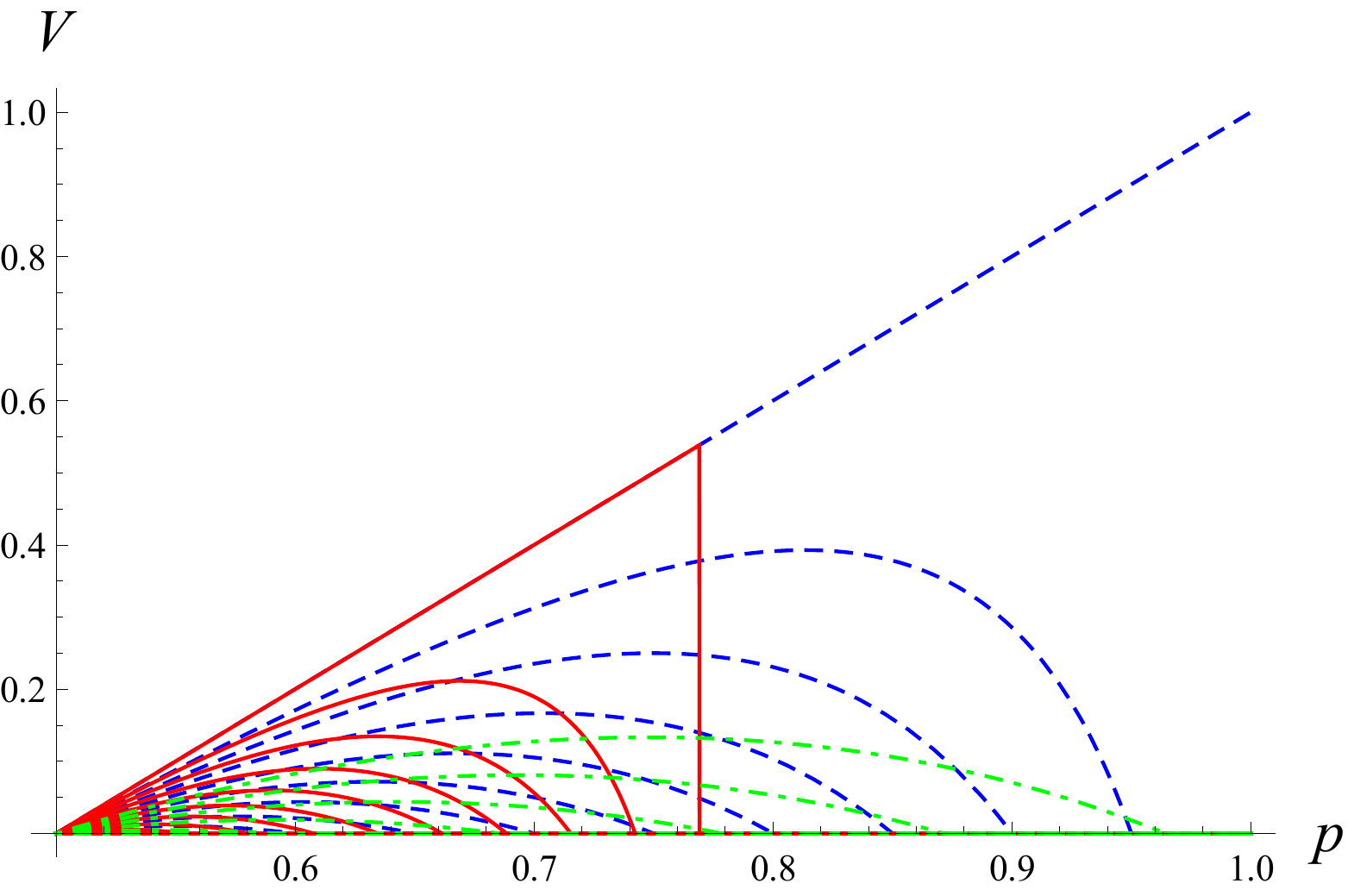}
\caption{Drift for $\rho=0$ (blue, dashed), $\rho=0.3$ (red,solid), and  $\rho=-0.3$ (green,dotdashed) as a function
  of $p$. From highest to lowest curves for $\alpha = 1, 0.95, \ldots,0.55$ (for $\rho=0$ and $\rho=0.3$), and for $\alpha = 0.75, 0.70,\ldots,0.55$ (for $\rho=-0.3$).}
\label{fig:rwre3Vversusp}
\end{figure}

\begin{figure}[H]
\centering
\includegraphics[width=0.8\linewidth]{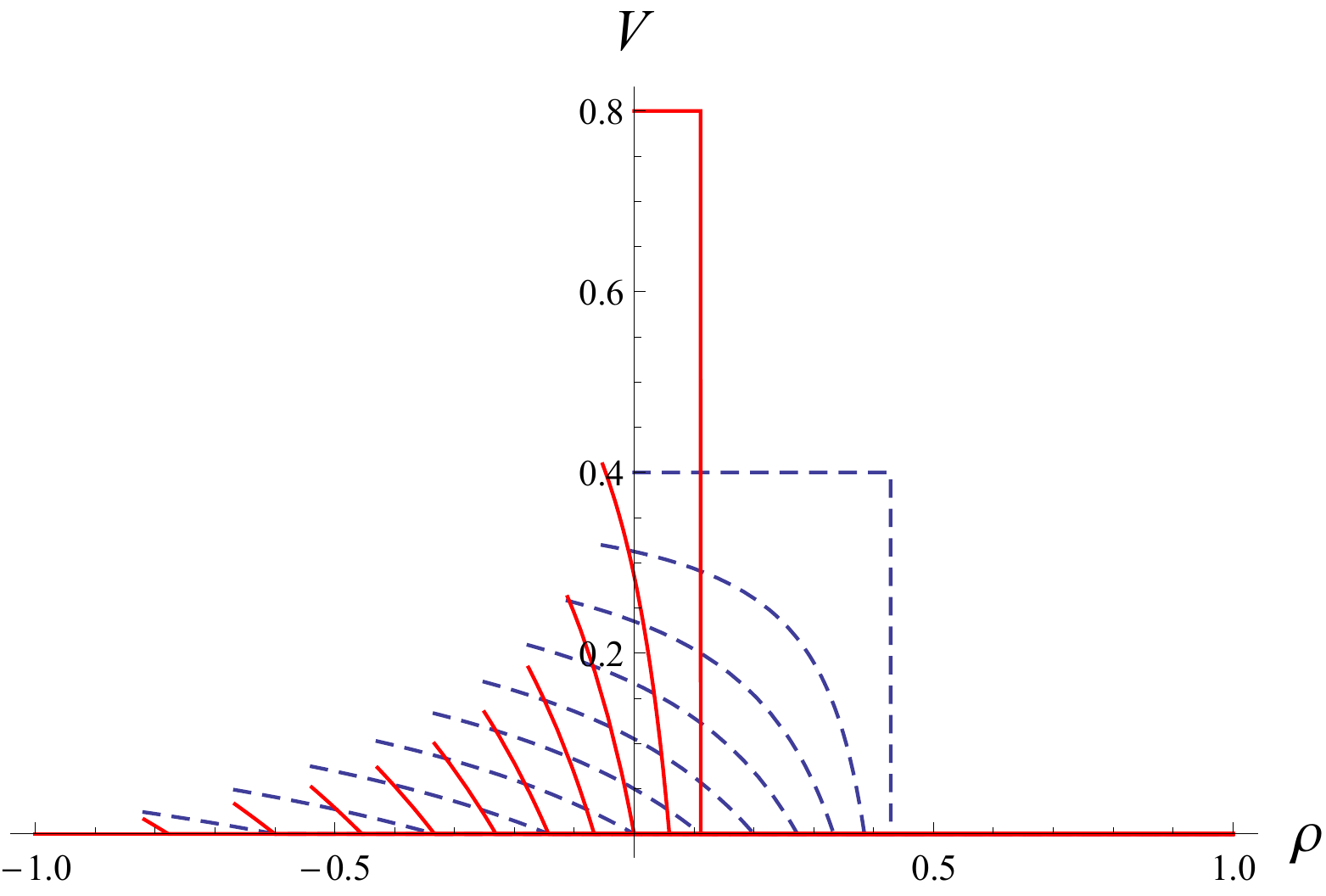}
\caption{Drift for $p=0.7$ (blue, dashed) and $p=0.9$ (red, solid) as a function of the correlation
  coefficient $\rho$, for $\alpha = 1, 0.95, \ldots,0.55$ (from
  highest to lowest curves). The values at $\rho = 0$ give the drift
  for the independent case. Note that $\rho$ must be greater than or
  equal to $1/\alpha$.}
\label{fig:rwre3Vversusrho}
\end{figure}

\subsection{2-dependent environment}  \label{sec:2dependence}

In this section we treat the $k$-dependent environment for $k=2$. For this case we have the transition probabilities
\[
P_{u_{i-2}u_{i-1},u_i}=\Pm(U_{i}=u_i \gvn U_{i-2}=u_{i-2}, U_{i-1}=u_{i-1}),\qquad u_j \in\{-1,1\},
\]
so that the one-step transition matrix of the Markov chain $\{Y_i, i \in \Z\}$ with $Y_i=(U_{i-1}, U_i)$ is given by
\[
P=\left[
\begin{array}{cccc}
P_{-1-1,-1}&P_{-1-1,+1}&0&0\\
0&0&P_{-1+1,-1}&P_{-1+1,+1}\\
P_{+1-1,-1}&P_{+1-1,+1}&0&0\\
0&0&P_{+1+1,-1}&P_{+1+1,+1}
\end{array}
\right]
=
\left[
\begin{array}{cccc}
1-a_-&a_-&0&0\\
0&0&b_-&1-b_-\\
1-a_+&a_+&0&0\\
0&0&b_+&1-b_+
\end{array}
\right].
\]
Thus, the model has five parameters,  $a_-, a_+, b_-, b_+,$ and $p$.
Also note that the special case $a_-=a_+(=a)$ and $b_-=b_+(=b)$ corresponds to the (1-dependent) Markovian case in Section~\ref{sec:markovdependence}.

We first note that the stationary distribution (row) vector $\vect{\pi}$ is given by
\begin{equation} \label{pi2dependent}
\vect{\pi}=\left(2+\frac{1-a_+}{a_-} + \frac{1-b_-}{b_+}\right)^{-1} \left(\frac{1-a_+}{a_-}, 1, 1, \frac{1-b_-}{b_+}\right),
\end{equation}
so assuming stationarity we have $\Pm(U_0=1)={\pi}_{-1,1}+{\pi}_{1,1}$
and $\Pm(U_0=-1)=\pi_{-1,-1}+\pi_{1,-1}$. It follows that
$\Pm(U_0=1)>\Pm(U_0=-1)$ if and only if
$\frac{a_-}{1-a_+}>\frac{b_+}{1-b_-}$. This is important to determine
the sign of $\Em[\log \sigma_0]$, which satisfies (with $\sigma=\frac{1-p}p$ as before), 
\[
\Em[\log \sigma_0] = \big(2\Pm(U_0=1)-1\big)\ \log \sigma.
\]
Hence, 
$X_n\rightarrow +\infty$ a.s.\ if and only if either $\frac{a_-}{1-a_+}>\frac{b_+}{1-b_-}$ and $p> 1/2$, or $\frac{a_-}{1-a_+}<\frac{b_+}{1-b_-}$ and $p < 1/2$;
$X_n\rightarrow -\infty$ a.s.\ if and only if either $\frac{a_-}{1-a_+}>\frac{b_+}{1-b_-}$ and $p < 1/2$, or 
$\frac{a_-}{1-a_+}<\frac{b_+}{1-b_-}$ and $p > 1/2$; and
$\{X_n\}$ is recurrent a.s.\ if and only if either $\frac{a_-}{1-a_+}=\frac{b_+}{1-b_-}$, or $p=1/2$, or both.

Next we consider the drift. As before we have when $\Em[S]<\infty$ that $V^{-1} =2\Em[S]-1$. So in view of (\ref{series}) we need to consider the matrix $PD$ where $D=\diag(\sigma^{-1}, \sigma, \sigma^{-1}, \sigma)$, so
\[
PD=
\left[
\begin{array}{cccc}
(1-a_-)\sigma^{-1}&a_-\sigma&0&0\\
0&0&b_-\sigma^{-1}&(1-b_-)\sigma\\
(1-a_+)\sigma^{-1}&a_+\sigma&0&0\\
0&0&b_+\sigma^{-1}&(1-b_+)\sigma
\end{array}
\right]
\]
and hence
\begin{eqnarray*}
V^{-1}&=& 2 \vect{\pi} \left(\sum_{n=0}^\infty (PD)^n\right) \ \vect{1}\ -1\\
&=& 2 \vect{\pi} (I-PD)^{-1} \ \vect{1}\ -1
\end{eqnarray*}
if Sp$(PD)<1$. Unfortunately, the eigenvalues of $PD$ are now the
roots of a 4-degree polynomial, which are hard to find
explicitly. However, using Perron--Frobenius theory and the implicit
function theorem it is possible to prove the following lemma, which has the same structure as in the Markovian case.

\begin{lemma}  \label{lem2dep}
The matrix series $\sum_{n=0}^\infty (PD)^n$ converges to
$(I-PD)^{-1}$, which is  
{\tiny
\[
\begingroup
\renewcommand{\arraystretch}{2.5}
\begin{bmatrix}
 1 - a_+ b_- - \sigma + \sigma B & a_-  \sigma ((b_+ -1) \sigma+1) & a_-  (-\sigma b_-+b_-+b_+ 
   \sigma) & -a_-  (b_--1) \sigma^2 \\
 \frac{(a_+ -1) (b_- (\sigma-1)-b_+  \sigma)}{\sigma^2} & \frac{(a_- +\sigma-1) ((b_+ -1) \sigma+1)}{\sigma} & -\frac{(a_- +\sigma-1)
   (b_- (\sigma-1)-b_+  \sigma)}{\sigma^2} & -(b_--1) (a_- +\sigma-1) \\
 -\frac{(a_+ -1) ((b_+ -1) \sigma+1)}{\sigma} & (a_- +a_+  (\sigma-1)) ((b_+ -1) \sigma+1) & \frac{(a_- +\sigma-1) ((b_+ -1)
   \sigma+1)}{\sigma} & -(b_--1) (a_- +a_+  (\sigma-1)) \sigma \\
 \frac{b_+ -a_+  b_+ }{\sigma^2} & \frac{b_+  (a_- +a_+  (\sigma-1))}{\sigma} & \frac{b_+  (a_- +\sigma-1)}{\sigma^2} &
   \frac{1 - A  + \sigma - a_+ b_- \sigma}{\sigma} \\
\end{bmatrix}
\endgroup
\]
}
\!\!divided by $\det(I -PD)=-\sigma^{-1}(\sigma-1)\big((1-B)\sigma-(1-A)\big)$,
iff $\sigma$ lies between 1 and~$\frac{1-A}{1-B}$. Here, $A=a_-+a_+b_--a_-b_-$ and $B=b_++a_+b_--a_+b_+$.
\end{lemma}
\begin{proof}
To find out for which values of $\sigma$ we have Sp$(PD)<1$, first we denote the (possibly complex) eigenvalues  of $PD$  by  $\lambda_i(\sigma), i=0, 1, 2, 3, $ as continuous functions of $\sigma$.
Since $PD$ is a nonnegative irreducible matrix for any $\sigma>0$, we can apply Perron--Frobenius to claim that there is always a unique eigenvalue with largest absolute value (the other $|\lambda_i|$ being strictly smaller), and that this eigenvalue is real and positive (so in fact it always equals Sp$(PD)$). When $\sigma=1$ the matrix is stochastic and we know this eigenvalue to be 1, and denote it by $\lambda_0(1)$. 

Now, moving $\sigma$ from 1 to any other positive value, $\lambda_0(\sigma)$ {\em must} continue to play the role of the Perron--Frobenius eigenvalue; i.e., none of the other $\lambda_i(\sigma)$ can at some point take over this role. If this were not true, then the continuity of the $\lambda_i(\sigma)$ would imply that one value $\hat \sigma$ exists where (say) $\lambda_1$ `overtakes' $\lambda_0$, meaning that $|\lambda_1(\hat \sigma)|=|\lambda_0(\hat \sigma)|$, which is in contradiction with the earlier Perron--Frobenius statement.

Thus, it remains to find out when $\lambda_0(\sigma)<1$, which can be established using the implicit function theorem, since 
$\lambda_0$ is implicitly defined as a function of $\sigma$ by $f(\sigma, \lambda_0)=0$, with $f(\sigma, \lambda)=\det(\lambda I-PD)$ together with $\lambda_0(1)=1$. Using $\det(D)=1$, we find that
\begin{align*}
f(\sigma, \lambda)=&\det((\lambda D^{-1}-P)D) =\det(\lambda D^{-1}-P)=\\
=&\ \sigma[\lambda(a_+b_--a_+b_+)+\lambda^3(b_+-1)]\\
&+\sigma^{-1}[\lambda(a_+b_--a_-b_-)+\lambda^3(a_--1)]\\
&+\lambda^4+(1-a_--b_++a_-b_+-a_+b_-)\lambda^2+a_-b_--a_-b_+-a_+b_-+a_+b_+.
\end{align*}
Setting $\lambda=1$ in this expression gives $\det(I-PD)$ as given in the lemma, with two roots for $\sigma$. Thus, there is only an eigenvalue 1 when $\sigma=1$, which we already called $\lambda_0(1)$, or when $\sigma=\frac{1-A}{1-B}$. In the latter case this must be $\lambda_0(\frac{1-A}{1-B})$, i.e., it cannot be $\lambda_i(\frac{1-A}{1-B})$ for some $i\neq0$, again due to continuity. As a result we have either $\lambda_0(\sigma)>1$ or $\lambda_0(\sigma)<1$ when $\sigma$ lies between 1 and~$\frac{1-A}{1-B}$. Whether $\frac{1-A}{1-B}<1$ or $\frac{1-A}{1-B}>1$ depends on the parameters:
\begin{equation} \label{sigmalocation}
\frac{1-A}{1-B}>1\quad\Leftrightarrow\quad\frac{a_-}{1-a_+}<\frac{b_+}{1-b_-},
\end{equation}
where we used that $1-B=1-b_+-a_+b_-+a_+b_+>(1-b_+)(1-a_+)>0$. Now we
apply the implicit function theorem:
\begin{align}
\frac{\di\lambda_0(\sigma)}{\di\sigma}\Big|_{\sigma=1}
&=
-\left.
\frac{\ \
\frac{\partial f(\sigma,\lambda_0)}{\partial \sigma}\ \ 
}{
\frac{\partial f(\sigma,\lambda_0)}{\partial \lambda_0}
}\right|_{\sigma=1, \lambda_0=1}
\\
&=
-\frac{b_+(1-a_+)-a_-(1-b_-)}{a_-(1-b_-+b_+)+b_+(1-a_++a_-)}\\
&=
\frac{\frac{a_-}{1-a_+}-\frac{b_+}{1-b_-}}{\frac{a_-}{1-a_+}\left(1+\frac{b_+}{1-b_-}\right) +\frac{b_+}{1-b_-}\left(1+\frac{a_-}{1-a_+}\right) },
\end{align}
which due to (\ref{sigmalocation}) is $<0$ iff $\frac{1-A}{1-B}>1$ and
is $>0$ iff $\frac{1-A}{1-B}<1$, so that indeed
Sp$(PD)=\lambda_0(\sigma)<1$ if and only if $\sigma$ lies between 1
and $\frac{1-A}{1-B}$.
 
\end{proof}

Note that for the case  $\frac{a_-}{1-a_+}=\frac{b_+}{1-b_-}$ the
series never converges, as there is no drift,
$\Pm(U_0=1)=\Pm(U_0=-1)$. This corresponds to $a=b$ in the Markovian case and $\alpha=1/2$ in the iid case.

We conclude that if $\sigma$ lies between 1 and $\frac{1-A}{1-B}$, or equivalently, if $p$ lies between 1/2 and $\frac{1-B}{1-A+1-B}$, the drift is given by  $V=(2 \vect{\pi} (I-PD)^{-1} \ \vect{1}\ -1)^{-1}$, where $\vect{\pi}$ is given in (\ref{pi2dependent}) and  $(I-PD)^{-1}$ follows from Lemma~\ref{lem2dep}. Using computer algebra, this can be shown to equal
\begin{equation}  \label{2depdrift}
V=(2 p-1)\, \frac{d\, p(1-p)\big((1-B)(1-p)-(1-A)p\big) }
{\sum_{i=0}^3 c_i \, p^i}
\end{equation} 
where
\[
\begin{split}
d &= a_-  (b_- - b_+ -1)+ b_+(a_+ -a_- -1) \\
c_0 &= 2 a_- b_+  (b_-  - b_+) \\
c_1 & =-c_0( 2+ a_+ +a_-)+(B-A)(1-B)\\ 
c_2 & =-c_0-c_1-c_3\\
c_3 & = (B-A) (2-A-B).
\end{split}
\] 
Including the transience/recurrence result from the first part of this section, and including the cases with negative drift, we obtain the following analogon to Theorems~\ref{thm:iid} and~\ref{thm:general}.

\begin{thm} \label{thm:2-dependent}
We distinguish between transient cases with and without drift, and the recurrent case in the same way as for the Markov environment in Theorem~\ref{thm:general}. In particular, all statements (1a.), \ldots, (3) in Theorem~\ref{thm:general} also hold for the 2-dependent environment if we replace $a$ and $b$ by $A$ and $B$ respectively,  (\ref{Markovdrift}) by (\ref{2depdrift}), and (\ref{Markovdrift2}) by minus the same expression (\ref{2depdrift}) but with $p$ replaced by $1-p$.
\end{thm}




\subsubsection{Comparison with the Markov environment}
To facilitate a comparison between the drifts for the 
two-dependent and Markov
environments it is convenient to write the  
probability distribution vector of
$(U_0,U_1,U_2)$ as $\vect{\pi} R$, where  $\vect{\pi}$ is the distribution vector of $(U_0,U_1)$, see (\ref{pi2dependent}), and
\[
R = 
\left[
\begin{array}{cccccccc}
1-a_-&a_-&0&0&  0 & 0 & 0 & 0 \\
0&0&b_-&1-b_- & 0 & 0 & 0 & 0\\
0 & 0 & 0 & 0 & 1-a_+&a_+&0&0\\
0 & 0 & 0 & 0& 0&0&b_+&1-b_+
\end{array}
\right].
\]
Thus, $\vect{\pi} R=c \left(\frac{(1-a_+)(1-a_-)}{a_-}, 1-a_+, b_-,1-b_-,1-a_+, a_+, 1-b_-, \frac{(1-b_-)(1-b_+)}{b_+}\right),$ where $c=\left(2+\frac{1-a_+}{a_-} + \frac{1-b_-}{b_+}\right)^{-1}$. If we also define 
\[
M_0 = 
\begin{bmatrix}
1 & 0 \\
1 & 0 \\
1 & 0 \\
1 & 0 \\
0 & 1 \\
0 & 1 \\
0 & 1 \\
0 & 1 \\
\end{bmatrix},
M_{01} = 
\begin{bmatrix}
1 & 0 & 0 & 0 \\
1 & 0 & 0 & 0 \\
0 & 1 & 0 & 0 \\
0 & 1 & 0 & 0 \\
0 & 0 & 1 & 0\\
0 & 0 & 1 & 0\\
0 & 0 & 0 & 1\\
0 & 0 & 0 & 1\\
\end{bmatrix}, 
M_{02} = 
\begin{bmatrix}
1 & 0 & 0 & 0 \\
0 & 1 & 0 & 0 \\
1 & 0 & 0 & 0 \\
0 & 1 & 0 & 0 \\
0 & 0 & 1 & 0\\
0 & 0 & 0 & 1\\
0 & 0 & 1 & 0\\
0 & 0 & 0 & 1\\
\end{bmatrix}, 
\]
then the probability distribution vector of $U_0$, $(U_0,U_1)$, and
$(U_0,U_2)$ are respectively given by  
{\footnotesize
\[
\begin{split}
\vect{\pi} R M_0&=c \left(\frac{1-a_+}{a_-}+1, \frac{1-b_-}{b_+}+1\right),\\
\vect{\pi} R M_{01}=\vect{\pi}&=c \left(\frac{1-a_+}{a_-}, 1, 1, \frac{1-b_-}{b_+}\right),\\
\vect{\pi} R M_{02}&=c \left(\frac{(1-a_+)(1-a_-)}{a_-}+b_-,2-a_+- b_-,2-a_+- b_-, a_+ + \frac{(1-b_-)(1-b_+)}{b_+}\right).
\end{split}
\]
} 
\!\!\!Various characteristics of the distribution of $(U_0,U_1,U_2)$ are now
easily found. In particular, 
the probability $\Pm(U_0=1)$ is 
\[
\alpha = \frac{a_- (1-b_- +b_+ )}
{a_- (1-b_- +b_+ )   + b_+(1-a_+ +a_-)},  
\]
the correlation coefficient between $U_0$ and $U_1$ is 
\[
\rho_{01} = 1-\frac{a_- }{a_- +1-a_+}-\frac{b_+}{ b_++ 1-b_-}, 
\]
the correlation coefficient between $U_0$ and $U_2$ is 
\[
\begin{split}
\rho_{02} &= 1-(2-a_+-b_-)\left(
\frac{a_- }{a_- +1-a_+}+\frac{b_+}{ b_+ +1-b_-}
\right)\\
&=1-(2-a_+-b_-)(1-\rho_{01}),  
\end{split}
\]
and  $\Em[U_0 U_1 U_2]$ is 
\[
e_{012} = \frac{4 a_-  b_+(b_- -a_+) + 
a_- (1-b_- +b_+ )   - b_+(1-a_+ +a_-)\ }
{\phantom{4 a_-  b_+(b_- -a_+) + }\ a_- (1-b_- +b_+ )   + b_+(1-a_+ +a_-)\ }\;.
\]
The original parameters can  be expressed in terms of $\alpha,
\rho_{01}, \rho_{02}$, and $e_{012}$  as follows: 
\[
\begin{split}
a_- & =   -\frac{2 \alpha (2 \alpha (\rho_{02}-1)-2
   \rho_{02}+1)+e_{012}+1}{8 (\alpha-1) (\alpha
   (\rho_{01}-1)+1)}\\
b_- & = \frac{2 \alpha (\alpha (4
   \rho_{01}-2 (\rho_{02}+1))-4 \rho_{01}+2 \rho_{02}+1)+e_{012}+1}{8
   (\alpha-1) \alpha (\rho_{01}-1)}\\
a_+ & = -\frac{2
   \alpha (2 \alpha (-2 \rho_{01}+\rho_{02}+1)+4 \rho_{01}-2
   \rho_{02}-3)+e_{012}+1}{8 (\alpha-1) \alpha
   (\rho_{01}-1)}\\
b_+ & = \frac{2 \alpha (-2 \alpha
   (\rho_{02}-1)+2 \rho_{02}-3)+e_{012}+1}{8 \alpha (\alpha
   (\rho_{01}-1)-\rho_{01})}\;.
\end{split}
\]
Note that due to the restriction that $a_-, a_+, b_-$, and $b_+$ are
probabilities, $(\alpha, \rho_{01},$  $\rho_{02},e_{012})$ can only take
values in a strict subset of $[0,1]\times [-1,1]^3$.

An illustration of the different behavior that can be achieved for
two-dependent environments (as opposed to Markovian environments) is
given in Figure~\ref{fig:twodepvsMarkov}. Here, $\alpha = 0.95$ and
$\rho_1 = 0.3$. The drift for the corresponding Markovian case is
indicated in the figure. 
The cutoff value is here approximately 0.75. By varying $\rho_2$ and
$e_{012}$ one can achieve a considerable increase in the drift. It is not
difficult to verify that the smallest
possible value for $\rho_2$ is here  $(\alpha - 1)/\alpha =  -1/19$,
in which case  $e_{012}$ can only take the value 
$ 3 + 2 \alpha (-5 - 4 \alpha(-1 + \rho_1) + 4 \rho_1) = 417/500.$
This gives a maximal cutoff value of 1. The corresponding drift curve
is indicated by the ``maximal'' label in Figure~\ref{fig:twodepvsMarkov}. 
For $\rho_2 = 0$, the parameter $e_{012}$ can at most vary from 
$-1 + 2 \alpha(-1 + \alpha(2 - 4 \rho_1) + 4 \rho_1) = 103/123 =
0.824$ to $7 + 2 \alpha (-9 + \alpha (6 - 4 \rho_1) + 4 \rho_1) =
211/250= 0.844$. The solid red curves show the evolution of the drift
between these extremes. The dashed blue curve corresponds to the drift for the
independent case with $\alpha = 0.95$.

\begin{figure}[H]
\centering
\includegraphics[width=0.8\linewidth]{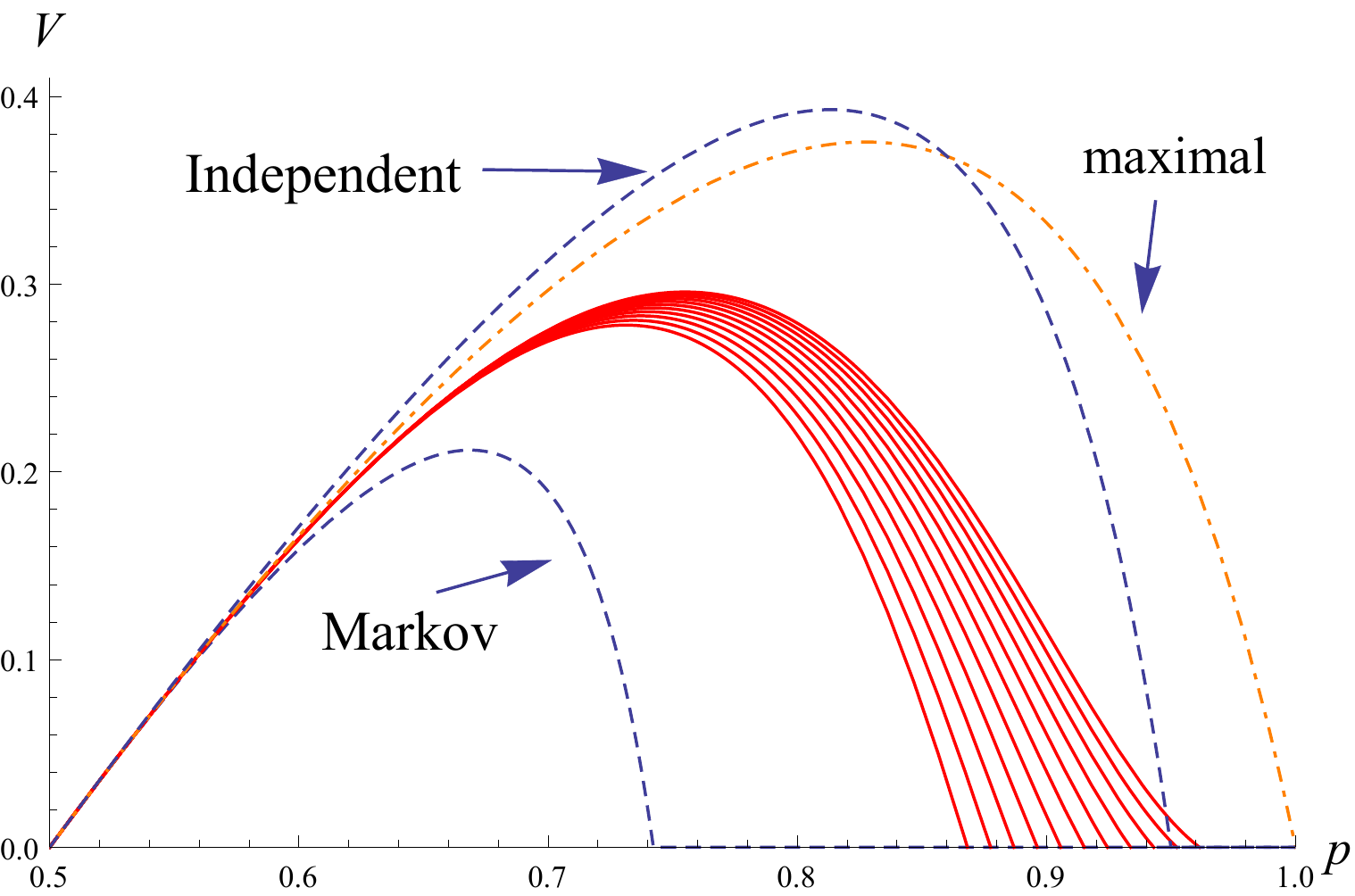}
\caption{Drift for $\alpha = 0.95$ and $\rho_1 = 0.3$ for various
  $\rho_2$ and $e_{012}$. The solid red curves show the drift for $\rho_2=0$
  and $e_{012}$ varying from 0.824 to 0.844. The smallest dashed blue curve 
  corresponds to the Markov case. The ``maximal'' dotdashed orange curve corresponds to the case $\rho_2
  =  - 1/19$ and $e_{012} = 417/500$. The middle dashed blue line gives the
  independent case.} 
\label{fig:twodepvsMarkov}
\end{figure}

\subsection{Moving average environment}
Recall that the environment is given by $U_i = g(Y_i)$ where the
Markov process $\{Y_i\}$ is given by $Y_i=(\hat U_{i}, \hat
U_{i+1},\hat U_{i+2})$. The sequence $\{\hat U_i\}$ is iid with
$\Pm(\hat U_i=1)=\alpha=1-\Pm(\hat U_i=-1)$. Thus, $\{Y_i\}$ has states
$1= (-1,-1,-1), 2 = (-1,-1,1), \ldots, 8 = (1,1,1)$ (in lexicographical
order) and transition matrix $P$ given by (\ref{P_movav}).  The
deterministic function $g$ is given by (\ref{g_movav}); see also
Figure~\ref{fig:movav}.

The almost sure behavior of $\{X_n\}$ again depends only on $\Em[U_0]$ which equals $-4\alpha^3+6\alpha^2-1=(2\alpha-1)(-2\alpha^2+2\alpha+1)$. Since $-2\alpha^2+2\alpha+1>0$ for $0\leq \alpha \leq 1$, the sign of $\Em[U_0]$ is the same as the sign of $\Em[\hat U_0]=2\alpha-1$, so the almost sure behavior is precisely the same as in the iid case; we will not repeat it here (but see Theorem~\ref{thm:movav}). 

To study the drift, we need the stationary vector of $\{Y_i\}$, which is given by
\begin{equation}  \label{pi_movav}
\begin{split}
\vect{\pi} & = 
\big\{(1-\alpha)^3,(1-\alpha)^2 \alpha,(1-\alpha)^2
   \alpha,(1-\alpha) \alpha^2, \\
   & \hspace{2cm} (1- \alpha)^2 \alpha,(1-\alpha) \alpha^2,(1-\alpha)
   \alpha^2,\alpha^3 \big\},
\end{split}
\end{equation} 
and the convergence behavior of $\sum (PD)^n$, with
$D=\diag(\sigma^{-1}, \sigma^{-1}, \sigma^{-1}, \sigma, \sigma^{-1},$ $ \sigma, \sigma, \sigma)$. This is given in the following lemma.

\begin{lemma}  \label{lem_movav}
The matrix series $\sum_{n=0}^\infty (PD)^n$ converges to
$(I-PD)^{-1}$ iff $\sigma$ lies between 1 and~$\sigma_{\mathrm{cutoff}}$, which is the unique root $\neq1$ of 
\begin{equation}
\begin{split}
\det(I-PD)=&-\frac{\alpha(1-\alpha)^2}{\sigma^3}+\frac{\alpha^2(1-\alpha)^2}{\sigma^2}-\frac{(1-\alpha)(1-\alpha+\alpha^2)}{\sigma}+1
\\ & -2\alpha^2(1-\alpha)^2
-\alpha^2(1-\alpha)\sigma^3+\alpha^2(1-\alpha)^2\sigma^2-\alpha(1-\alpha+\alpha^2)\sigma.  \label{Detmovav}
\end{split} 
\end{equation}
\end{lemma}
\begin{proof}
The proof is similar to that of Lemma~\ref{lem2dep}; we only give an outline, leaving details for the reader to verify. Again, denote the possibly complex eigenvalues  of $PD$  by  $\lambda_i(\sigma), i=0, \ldots, 7$ and use Perron-Frobenius theory to conclude that for any $\sigma>0$ we have Sp$(PD)=\lambda_0(\sigma)$, say, with $\lambda_0(1)=1$.

To find out when $\lambda_0(\sigma)<1$ we again use the implicit function theorem on $f(\sigma, \lambda_0)=0$, with $f(\sigma, \lambda)=\det(\lambda I-PD)$. 
Setting $\lambda=1$ gives (\ref{Detmovav}). It can be shown that $f(\sigma,1)$ is zero at $\sigma=1$, that $f(\sigma,1)\ra \infty$ for $\sigma\downarrow 0$,  and that $(\partial^2/\partial \sigma^2) f(\sigma,1)<0$ for all $\sigma>0$ (for the latter, consider $0<\sigma<1$ and $\sigma\geq1$ separately). Thus we can conclude that 
$f(\sigma,1)$ has precisely two roots for $\sigma>0$, at $\sigma=1$ and at $\sigma=\sigma_{\mathrm{cutoff}}$.

As a result we have either $\lambda_0(\sigma)>1$ or $\lambda_0(\sigma)<1$ when $\sigma$ lies between 1 and~$\sigma_{\mathrm{cutoff}}$. For the location of $\sigma_{\mathrm{cutoff}}$ it is helpful to know that 
 $(\partial/\partial \sigma) f(\sigma,1)\big|_{\sigma=1}=(2\alpha-1)(2\alpha^2-2\alpha-1)$, which is positive for $0<\alpha<1/2$ and negative for $1/2<\alpha<1$. Thus we have $\sigma_{\mathrm{cutoff}}>1$ iff $\alpha<1/2$.
Also $(\partial/\partial \lambda) f(1,1)=1$ so that the implicit function theorem gives $(d/d\sigma) \lambda_0(\sigma)\big|_{\sigma=1}=-(2\alpha-1)(2\alpha^2-2\alpha-1)$, so that indeed $\lambda_0(\sigma)<1$
iff $\sigma$ lies between 1 and~$\sigma_{\mathrm{cutoff}}$. 
\end{proof}


\medskip
\noindent
The cutoff value for $p$ is now easily found as $(1+\sigma_{\mathrm{cutoff}})^{-1}$, which can be numerically evaluated. The values are plotted in Figure~\ref{fig:movavcutoff}.
\begin{figure}[H]
\centering
\includegraphics[width=0.75\linewidth]{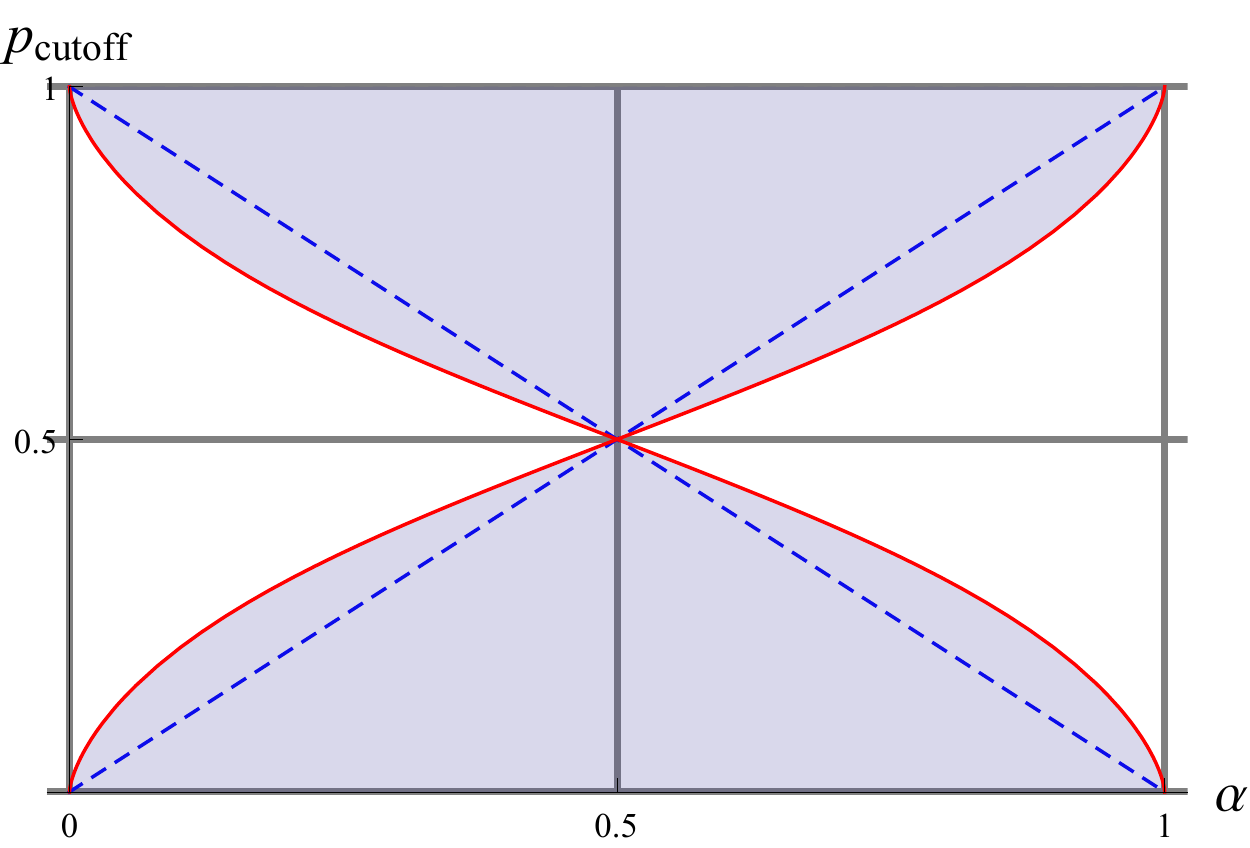}
\caption{Relation between cutoff value for $p$, and $\alpha$. The solid red curve is for the moving average process. For comparison, the dashed blue line is the iid case (see also
Figure~\ref{fig:iidcase}).} 
\label{fig:movavcutoff}
\end{figure}

When $p$ lies between 1/2 and $p_{\mathrm{cutoff}}$, the drift is
given by $V=(2 \vect{\pi} (I-PD)^{-1} \ \vect{1}\ -1)^{-1}$, where
$\vect{\pi}$ is given in (\ref{pi_movav}) and  $(I-PD)^{-1}$ follows
from Lemma~\ref{lem_movav}. Using computer algebra we can find a rather unattractive, but explicit expression for the value of the drift; it is given by
%
%
 the quotient of
\[
\begin{split}
& \alpha ^4 \left(-(1-2 p)^2\right) (p-1) p+\alpha ^3 (1-2 p ((p-2) p (p
(2 p-5)+6)+4)) \\
& +\alpha ^2 (2 p-1) (p (3 p ((p-2)
   p+3)-5)+1)-\alpha  (1-2 p)^2 p^2-(p-1)^2 p^3 (2 p-1)
\end{split}
\] 
and 
\[
\begin{split}
& -2 \alpha ^5 (2 p-1)^3-\alpha ^4 (1-2 p)^2 ((p-11) p+6)+\alpha ^3 (2
  p-1) (2 p (p^3-9 p+10) \\
 & -5)-\alpha ^2 (p+1)
   (2 p-1) (p (p (3 p-7)+6)-1)+\alpha  p^2 (2 p-1)+(p-1)^2 p^3\;.
\end{split}
\]

\begin{thm}   \label{thm:movav}
Let $p_{\mathrm{cutoff}}=(1+\sigma_{\mathrm{cutoff}})^{-1}$, where
$\sigma_{\mathrm{cutoff}}$ follows from Lemma \ref{lem_movav}. Then $p_{\mathrm{cutoff}}>1/2$ iff $\alpha>1/2$.
We distinguish between transient cases with and without drift, and the recurrent case as follows.
\begin{enumerate}
\item[1a.] If either $\alpha > 1/2$ and $p\in (1/2, p_{\mathrm{cutoff}})$ or $\alpha < 1/2$ and $p
  \in (p_{\mathrm{cutoff}}, 1/2)$, then almost surely $\ds \lim_{n\ra \infty} {X_n} = \infty\;$ and the drift $V>0$ is given as above.
\item[1b.] If either 
 $\alpha > 1/2$ and $p \in (1-p_{\mathrm{cutoff}},1/2)$ or 
  $\alpha < 1/2$ and $p \in (1/2, 1-p_{\mathrm{cutoff}})$, then  almost surely $\ds \lim_{n\ra \infty} {X_n} = -\infty\;$ and the drift $V<0$ is given as minus the same expression as above but with $p$ replaced by $1-p$.
\item[2a.] If either $\alpha > 1/2$ and $p\in [p_{\mathrm{cutoff}},1]$ or $\alpha < 1/2$ and $p
  \in [0, p_{\mathrm{cutoff}}]$, then almost surely $\ds \lim_{n\ra \infty} {X_n} = \infty\;,$ but $V=0$.
\item[2b.] If either $\alpha > 1/2$ and $p\in [0, 1-p_{\mathrm{cutoff}}]$ or $\alpha < 1/2$ and $p
  \in [1-p_{\mathrm{cutoff}}, 1]$, then almost surely $\ds \lim_{n\ra \infty} {X_n} = -\infty\;,$ but $V=0$.
\item[3.] Otherwise (when $\alpha = 1/2$ or $p=1/2$ or both),  $\{X_n\}$ is recurrent and $V = 0$. 
\end{enumerate}
\end{thm}

Figure~\ref{fig:movav_vs_indep} compares the drifts for the moving
average and independent environments. 

\begin{figure}[H]
\centering
\includegraphics[width=0.85\linewidth]{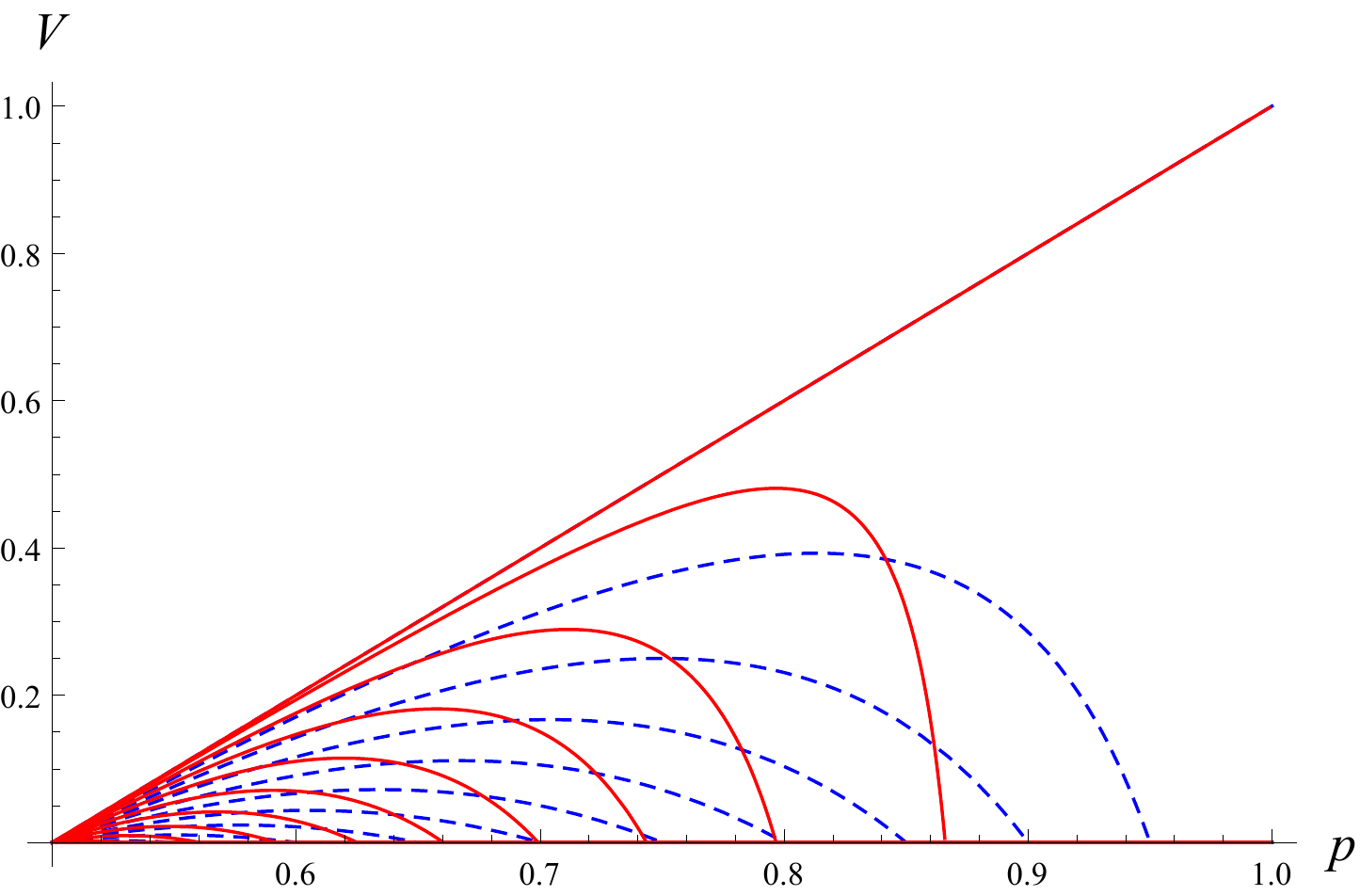}
\caption{Red: Drift for the moving average environment as a function of $p$ 
 for $\alpha = 1, 0.95, \ldots,0.55$ (from
  highest to lowest curves). Blue: comparison with the independent
  case. }
\label{fig:movav_vs_indep}
\end{figure}

It is interesting to note that the cutoff points
(where $V$ becomes 0) are significantly lower in the moving average
case than the iid case, using the same $\alpha$, while 
at the same time
the maximal drift that can be achieved is  {\em higher} for the moving
average case than for the iid case. This is different behavior from
the Markovian case; see also Figure~\ref{fig:rwre3Vversusp}.

\section{Conclusions}\label{sec:concl}
Random walks in random environments can exhibit interesting and unusual
behavior due to the trapping phenomenon. The dependency structure of
the random environment can significantly affect the drift of the
process. We showed how to conveniently construct dependent environment
processes, including $k$-dependent and moving average environments,
by using an auxiliary Markov chain. For the well-known swap RWRE
model,  
this approach allows for easy
computation of drift, as well as explicit 
conditions under which the drift is
positive, negative, or zero. The cutoff values where the drift
becomes zero, are determined via Perron--Frobenius theory.
Various  generalizations of the above environments can be considered in the same (swap model) framework, and can be analyzed along the same lines, e.g.,  replacing iid by Markovian $\{\hat{U}_i\}$ in the
moving average model, or taking moving averages of more than 3
neighboring states. 

\enlargethispage{\baselineskip}

Other possible directions for future research are (a) extending the
two-state dependent random environment to a $k$-state dependent random
environment; (b) replacing the transition probabilities for swap model
with the more general rules in Eq.\eqref{transitions}; and  (c) generalizing
the single-state random walk process to a multi-state discrete-time
{\em quasi birth and death process}
(see, e.g., \cite{bean1997}). By using an infinite ``phase space'' for
such processes, it might be possible to bridge the gap between the
theory for one- and multi-dimensional RWREs.

\section*{Acknowledgements}
This work was supported by the Australian Research
Council {\em Centre of Excellence for Mathematical and Statistical
  Frontiers} (ACEMS) 
under grant number CE140100049. Part of this work was done while the
first author was an Ethel Raybould Visiting Fellow at The University
of Queensland. We thank Prof.\ Frank den Hollander for his useful
comments. 
\bibliographystyle{plain}
\bibliography{refs}

\begin{thebibliography}{10}

\bibitem{alili99}
S.~Alili.
\newblock Asymptotic behaviour for random walks in random environments.
\newblock {\em J. Appl. Prob.}, 36:334--349, 1999.

\bibitem{bean1997}
N.~G. Bean, L.~Bright, G.~Latouche, C.~E.~M. Pearce, P.~K. Pollett, and P.~G.
  Taylor.
\newblock The quasi-stationary behavior of quasi-birth-and-death processes.
\newblock {\em The Annals of Applied Probability}, 7(1):134--155, 02 1997.

\bibitem{brereton12}
T.~Brereton, D.P. Kroese, O.~Stenzel, V.~Schmidt, and B.~Baumeier.
\newblock Efficient simulation of charge transport in deep-trap media.
\newblock In C.~Laroque, J.~Himmelspach, R.~Pasupathy, O.~Rose, and A.~M.
  Uhrmacher, editors, {\em Proceedings of the 2012 Winter Simulation
  Conference}, Berlin, 2012.

\bibitem{chernov67}
A.~A. Chernov.
\newblock Replication of multicomponent chain by the lighting mechanism.
\newblock {\em Biophysics}, 12:336--341, 1967.

\bibitem{dolgopyat2008}
D.~Dolgopyat, G.~Keller, and C.~Liverani.
\newblock Random walk in {Markovian} environment.
\newblock {\em The Annals of Probability}, 36(5):1676--1710, 09 2008.

\bibitem{greven94}
A.~Greven and F.~den Hollander.
\newblock Large deviations for a random walk in random environment.
\newblock {\em Ann. Probab.}, 22:1381--1428, 1994.

\bibitem{hughes2}
B.~D. Hughes.
\newblock {\em Random Walks and Random Environments}.
\newblock {Oxford University Press}, 1996.

\bibitem{kesten75}
H.~Kesten, M.~W. Koslow, and F.~Spitzer.
\newblock A limit law for random walk in a random environment.
\newblock {\em Compositio Math.}, pages 145--168, 1975.

\bibitem{kozlov85}
S.~M. Kozlov.
\newblock The method of averaging and walks in inhomogeneous enviroments.
\newblock {\em Russian Math. Surveys}, 40:73--145, 1985.

\bibitem{wolf04}
E.~Mayer-Wolf, A.~Roitershtein, and O.~Zeitouni.
\newblock Limit theorems for one-dimensional transient random walks in {Markov}
  environments.
\newblock {\em Ann. Inst. H. Poincar\'e Probab. Statist.}, 40(5):635--659,
  2004.

\bibitem{revesz}
P.~R\'ev\'esz.
\newblock {\em Random Walk in Random and Non-Random Environments}.
\newblock World Scientific, third edition, 2013.

\bibitem{sinai83}
Y.~G. Sinai.
\newblock The limiting behavior of a one-dimensional random walk in a random
  medium.
\newblock {\em Theory Prob. Appl.}, 27(2):256--268, 1982.

\bibitem{solomon75}
F.~Solomon.
\newblock Random walks in a random environment.
\newblock {\em Ann. Prob.}, 3:1--31, 1975.

\bibitem{stenzel14}
O.~Stenzel, C.~Hirsch, V.~Schmidt, T.~Brereton T., D.P. Kroese, B.~Baumeier,
  and D.~Andrienko.
\newblock A general framework for consistent estimation of charge transport
  properties via random walks in random environments.
\newblock {\em Multiscale Modeling and Simulation}, 2014.
\newblock Accepted for publication, pending minor revision.

\bibitem{sznitman04}
A.-S. Sznitman.
\newblock {\em Topics in random walks in random environment}, pages 203--266.
\newblock ICTP Lecture Notes Series, Trieste, 2004.

\bibitem{temkin69}
D.~E. Temkin.
\newblock A theory of diffusionless crystal growth.
\newblock {\em Kristallografiya}, 14:423--430, 1969.

\bibitem{zeit04}
O.~Zeitouni.
\newblock {\em Lecture notes on random walks in random environment}, volume
  1837 of {\em Lecture Notes in Mathematics}.
\newblock Springer, 2004.

\bibitem{zeit2012}
O.~Zeitouni.
\newblock Random walks in random environment.
\newblock In Robert~A. Meyers, editor, {\em Computational Complexity}, pages
  2564--2577. Springer, New York, 2012.

\end{thebibliography}

\end{document}